\documentclass[12pt,reqno]{amsart}
\usepackage{amsmath}
\usepackage{amssymb,amscd}

\numberwithin{equation}{section}

\newtheorem{theorem}{Theorem}[section]
\newtheorem{corollary}[theorem]{Corollary}
\newtheorem{lemma}[theorem]{Lemma}

\theoremstyle{definition}
\newtheorem{defn}[theorem]{Definition}
\newtheorem{remark}[theorem]{Remark}
\newtheorem{example}[theorem]{Example}

\def \bb{\mathbb}
\def \mb{\mathbf}
\def \mc{\mathcal}

\def \CC{{\bb{C}}}

\def \({\left(}
\def \){\right)}
\def \<{\langle}
\def \>{\rangle}

\def \Aut{{\rm Aut}}

\begin{document}

\title[Equivariant principal bundles over smooth toric  varieties]{A classification of equivariant
principal bundles over nonsingular toric varieties}

\author[I. Biswas]{Indranil Biswas}
\address{School of Mathematics, Tata Institute of Fundamental Research, Mumbai, India }

\email{indranil@math.tifr.res.in}

\author[A. Dey]{Arijit Dey}

\address{Department of Mathematics, Indian Institute of Technology-Madras, Chennai, India }

\email{arijitdey@gmail.com}

\author[M. Poddar]{Mainak Poddar}

\address{Departamento de
Matem\'aticas, Universidad de los Andes, Bogot\'a, Colombia; 
and  \\ Department of Mathematics, Middle East Technical University, Northern Cyprus Campus, 
Kalkanli, Guzelyurt, KKTC, Mersin 10 Turkey}

\email{mainakp@gmail.com}



\begin{abstract}
We classify holomorphic as well as algebraic torus
equivariant principal $G$-bundles over a nonsingular toric variety
$X$, where $G$ is a  complex linear algebraic group. It is shown that any 
such bundle over an affine, nonsingular toric variety admits a trivialization
in equivariant sense. We also obtain some splitting results.
\end{abstract}

\maketitle

\section{Introduction}

Denote the algebraic torus $(\CC^*)^n$ by $T$. Let $N$ denote the group of $1$-parameter
subgroups of $T$; so $N$ is isomorphic to $\mathbb Z^n$. Let $M$ denote the group of
characters of $T$. Then $M \,=\, Hom_{\mathbb Z}(N,\mathbb Z)$ by restricting a
character to the $1$-parameter subgroups (the endomorphisms of $\CC^*$ is $\mathbb Z$).

Let $X$ be a complex manifold (resp. smooth complex algebraic variety) equipped with a holomorphic (resp. algebraic) left-action of $T$, and let $G$
be a complex linear  algebraic group. A $T$-equivariant holomorphic (resp. algebraic) principal $G$-bundle on $X$ is a
holomorphic (resp. algebraic) principal $G$-bundle $\pi: \mathcal{E} \to X$ with $T$ acting holomorphically (resp. algebraically)
on $\mathcal{E}$, such that $\pi$ is $T$-equivariant and
$$t(z\cdot g)\,=\, (tz)\cdot g \ ~ \ \forall \ t\in T,\,g \in G,\, z\in
\mathcal{E}$$
(this means that the actions of $T$ and $G$ on $\mathcal{E}$ commute).

Now consider $X= X_{\Xi}$ to be a nonsingular complex toric variety of dimension $n$
corresponding to a fan $\Xi$. Our main goal is to
give a description of isomorphism classes of 
holomorphic (resp. algebraic) $T$-equivariant principal 
$G$-bundles over
$X$; this is carried out in Theorem \ref{class}. This description is in terms of
equivalence classes of data $\{\rho_{\sigma}\, , P(\tau,\sigma)\}$, where
$\sigma$ runs over maximal cones of $\Xi$,
while $\rho_{\sigma}:T \to G$ is a holomorphic (equivalently algebraic) group homomorphism and
$P(\tau,\sigma)$ is a $G$-valued $1$-cocycle. This data
$\{\rho_{\sigma}, P(\tau,\sigma)\}$ is required to satisfy certain conditions.
In particular, the data depends on choice of projections $\pi_{\sigma}: T \to T_{\sigma} $ (see
 section \ref{sds}).  Here $T_{\sigma}$ denotes the stabilizer of the orbit corresponding 
 to the cone $\sigma$.  However, if every maximal cone is of top dimension (such as for a complete toric variety), then such a choice of  projections  is not necessary.

Note that the description of isomorphism classes is the same in both the algebraic 
and holomorphic cases. We prove that every isomorphism class of holomorphic 
bundles contains an algebraic representative; see Remark \ref{algrep}. 

As an application of the above description,
we prove that  if $G$ is nilpotent then the principal $G$-bundle
admits an equivariant reduction of structure group to a torus (meaning
the bundle is split). For $G$ abelian, this
splitting result appeared in \cite{DP}. 

In section \ref{ktd}, we give a more combinatorial description of the isomorphism classes
of $T$-equivariant  principal $G$-bundles
in terms of characters of $T$ and $G$-valued $1$-cocycles; see
Theorems \ref{class2} and \ref{class3}. These 
generalize  Kaneyama's description \cite{Kan1} of the isomorphism classes of $T$-equivariant vector bundles over
complete nonsingular toric varieties in two directions. Firstly, the assumption of completeness on the base is removed. Secondly, the  bundles are generalized
from vector to principal bundles. 

 If $G$ has a normal  maximal torus,
then a $T$-equivariant principal $G$-bundle  over $X$ has an equivariant reduction
 of structure group to the maximal torus; see Lemma \ref{normal}.
 Our treatment also makes clearer the approach of Kaneyama.  
 This helps us unearth a gap in the proof 
 of the splitting of small rank equivariant vector bundles over
  projective space in \cite{Kan2}; see Remark \ref{srb}.

In Lemma \ref{aut} we show that the automorphism group of a
$T$-equivariant holomorphic (or algebraic) principal $G$-bundle over a toric
variety is a subgroup of $G$ that contains the center of $G$.

 A crucial tool used by us is the following equivariant Oka principle due to
Heinzner and Kutzschebauch \cite{HK}. The algebraic analogue of this is derived
 in the toric case in Lemma \ref{hka}.
 
\begin{theorem}\label{hk}\cite{HK} Suppose $K$ is any complex
reductive group acting linearly on $\CC^k$. Then any
$K$-equivariant holomorphic principal $G$-bundle $\mc{E}$ over $\CC^k$ is
equivariantly isomorphic to $\CC^k \times \mc{E}(0)$ where the
latter has diagonal $K$-action. \end{theorem}

Some of the basic lemmas regarding local action functions in section 2
have appeared before in \cite{DP}. We include them
here for the  convenience of the reader.

In the topological case, the classification of $(S^1)^n$-equivariant principal $G$-bundles over a toric manifold, where $G$ is a compact Lie group, is contained in the results of Hambleton and Hausmann \cite{HH}. The classification of $(\mathbb{C}^*)^n$-equivariant  topological principal bundles over a toric manifold appears to be an interesting open problem.

\section{Distinguished sections}\label{sds}

Denote by $\Xi(d)$ the set of all
$d$-dimensional cones in $\Xi$. For a cone $\sigma$ in
$\Xi$, denote the corresponding affine variety and $T$-orbit
by $X_{\sigma}$ and $O_{\sigma}$ respectively.
Let $T_{\sigma}$ denote the stabilizer of any point in $O_{\sigma}$;
it is independent of the point because $T$ is abelian. We recall that each
$O_{\sigma}$ may be given a group structure using a specific identification
with $T / T_{\sigma} $, see \cite[p. 53]{Ful} or \cite[Proposition 1.6]{Oda}.  The principal orbit
$O$ may be identified with $T$. 

For each $\sigma$, fix once and for all, an embedding
$\iota_{\sigma}: O_{\sigma} \hookrightarrow T$ that splits 
the exact sequence $$ 1 \to T_{\sigma} \to T \to O_{\sigma} \to 1 \,. $$
As we are in the nonsingular case, such an embedding may be obtained by extending a 
set of independent, primitive generators of $\sigma$ to a basis of the lattice $N$. 
 A splitting, as above, induces a direct product
decomposition $T\,=\, T_{\sigma}\times \iota_{\sigma}(O_{\sigma})$.
Let $\pi_{\sigma}: T \to T_{\sigma} $ be the projection associated
to this decomposition of $T$.

Fix any $\sigma \,\in\, \Xi$.  Let $X \,=\, X_{\sigma}$.
There exists an affine toric variety $A_{\sigma}$ with the dense orbit isomorphic  to $T_{\sigma}$, such that
 $X$ is $T$-equivariantly isomorphic to $A_{\sigma}\times O_{\sigma}$. Under this isomorphism, the orbit
 $O_{\sigma}$ in $X$ is identified with $\{p\} \times O_{\sigma} $,
 where $p$ is the $T_{\sigma}$-fixed point of
$A_{\sigma}$. Moreover, there exist  isomorphisms
 $a_{\sigma} : T_{\sigma} \to (\CC^*)^{\dim(\sigma)}$ and
 $\phi_{\sigma} : A_{\sigma}  \to \CC^{\dim(\sigma)}$, satisfying
 $$\phi_{\sigma} (t x ) \,=\, a_{\sigma}(t) \phi_{\sigma}(x) \, $$ for all $x \in A_{\sigma}$ and  $t \in T_{\sigma}$.

Let $G$ denote a complex linear  algebraic group. Suppose $\mathcal{E}$ is a
$T$-equivariant holomorphic (resp. algebraic) principal $G$-bundle over $X$ 
which admits a holomorphic (resp. algebraic) trivialization.
Let $s: X \to \mathcal{E}$ be any holomorphic (resp. algebraic) section. We encode the
$T$-action on $\mathcal{E}$ as follows:

\begin{defn} For any $x\in X$ and $t\in T$, define $\rho_s(x,t) \in G$ by
$$
t s(x) \,=\, s(tx)\cdot \rho_s(x,t)\, .
$$
Since the action of
$G$ on each fiber of $\mathcal{E}$ is free and transitive, it follows that
$\rho_s(x,t)$ is well-defined and it is holomorphic (resp.  algebraic) in both $x$ and $t$.
We say that $\rho_s: X \times T \to G$ is the local action function associated to $s$.
\end{defn}

If $s'(x)\,=\, s(x)\cdot g(x)$ is another holomorphic (resp. algebraic) section of $\mathcal E$,
then it is straight-forward to check that
\begin{equation}\label{eq:rho2}
 \rho_{s'}(x,t) \,=\,  g(tx)^{-1} \rho_s (x,t) g(x)\, .
\end{equation}

\begin{lemma}\label{lem:rho3} For any $t_1\, ,t_2 \in\,T$, the equality
$$\rho_s(x, t_1t_2)\,=\, \rho_s(t_2x,t_1) \rho_s(x,t_2)$$ holds.
\end{lemma}

\begin{proof}
By the definition of $\rho_s$, we have $t_1t_2s(x)\,=\, s(t_1t_2x) \cdot \rho_s(x,t_1t_2)$.
On the other hand, $$t_1t_2s(x) = t_1( s(t_2x) \cdot \rho_s(x, t_2 )) \,=\,
(t_1 s(t_2x)) \cdot \rho_s(x, t_2 )  = s(t_1t_2x) \cdot
\rho_s(t_2x, t_1) \rho_s(x,t_2)\, .$$
The lemma follows from these.
\end{proof}

Suppose $\delta$ is any subcone of $\sigma$. Then Lemma \ref{lem:rho3} implies that for
any $x_{\delta} \,\in\, O_{\delta}$, the restriction
$$
\rho_s(x_{\delta}, \cdot)\,:\, T_{\delta} \,\to\, G
$$
is a group homomorphism, where $T_{\delta}$ as before is the stabilizer.

\begin{lemma}\label{lem:rho5} If $\rho_s(x,\cdot)$ is independent of $x$, then
it defines a group homomorphism $$\rho_s: T \to G\, .$$ Conversely if
$\rho_s(x_0,\cdot)$ is a group homomorphism for some $x_0 \,\in\, O$,
then $\rho_s(x,\cdot)$ is independent of $x$.
\end{lemma}

\begin{proof} The first part follows immediately from Lemma \ref{lem:rho3}.

To prove the second part, assume that $\rho_s(x_0,\cdot)$ is a group homomorphism
for some $x_0\,\in\, O$. We have
$$ \rho_s(x_0, t) \rho_s(x_0, u) \,=\, \rho_s(x_0,tu)\,=\,\rho_s(ux_0,t) \rho_s(x_0, u )$$
for all $u\, ,t \,\in\, T$. Therefore, for any $u \,\in\, T$, we have
$$ \rho_s(ux_0,\cdot) \,=\, \rho_s(x_0, \cdot).$$
Then, since $\rho_s$ is continuous, the second part of the lemma follows from the denseness
of $O$ in $X$.
\end{proof}

\begin{lemma}\label{lem:equiv} Let $X_1$ and $X_2$ be affine toric varieties.
 Let $\alpha: X_1 \to X_2$ be an isomorphism of $T$-spaces up to an
automorphism $a:T\to T$, meaning
$\alpha \circ t \,=\, a(t) \circ \alpha$. Suppose $\pi_i:
\mathcal{E}_i \to X_i$ is a $T$-equivariant trivial principal
$G$-bundle for $i=1,2$. Let $\phi: \mathcal{E}_1 \to \mathcal{E}_2$ be an
isomorphism of $T$-equivariant principal $G$-bundles over $X$, which is
compatible with $\alpha$ and $a$:
$$ \pi_2 \circ \phi = \alpha \circ \pi_1  \quad {\rm and} \quad  \phi \circ t = a(t) \circ \phi. $$
 Let $s_1$ be any section of
$\mathcal{E}_1$. Let $s_2$ be the section of $\mathcal{E}_2$
defined by $s_2(\alpha(x))\,=\, \phi(s_1(x))$ for any $x \in X_1$.
 Then $\rho_{s_1}(x, t ) \,=\, \rho_{s_2}(\alpha(x),a(t))$ for every $x\in X_1$, $t \in
 T$. In particular, if $\alpha$ and $a$ are both identity, then $\rho_{s_1} \,=\,
 \rho_{s_2}$.
\end{lemma}

\begin{proof} This follows from the following calculation:
$$ \begin{array}{l} s_2(a(t) \alpha(x))\cdot \rho_{s_2}(\alpha(x),a(t)) \\
 = a(t) s_2( \alpha(x)) =  a(t) \phi( s_1 (x) ) \\
 = \phi (t s_1(x) ) = \phi(s_1(tx)\cdot \rho_{s_1}(x,t)) \\
= \phi(s_1(tx)) \cdot \rho_{s_1}(x,t)  = s_2(\alpha(tx))\cdot \rho_{s_1}(x,t) \\
 = s_2(a(t) \alpha(x) ) \cdot \rho_{s_1}(x,t).  \end{array}$$
\end{proof}

\begin{defn} We say that a section $s$  of $\mc{E}$ is distinguished if
\begin{itemize}
\item $\rho_{s}(x, \cdot)$ is independent of $x$, and

\item $\rho_s(x,\cdot)$ factors through the projection $\pi_{\sigma}:T \to T_{\sigma}$.
\end{itemize}
\end{defn}

\begin{lemma}\label{ds2}
If $s$ is a distinguished section, then so is
$s\cdot g$ for every $g \in G$.
\end{lemma}

\begin{proof} The assertion follows from the fact that the actions of $T$ and $G$ commute.
Note that $\rho_{s\cdot g}(x,t)\,=\, g^{-1} \rho_s(x,t) g$ by \eqref{eq:rho2}.
\end{proof}

\begin{lemma} \label{hka} Consider any linear action of the torus $K=(\mathbb{C}^*)^k $
on $\mathbb{C}^k$. Suppose $\mc{E}$ is a $ K $-equivariant algebraic principal $G$-bundle over    $\mathbb{C}^k$. Then $\mc{E}$ is $K$-equivariantly isomorphic to 
$\mathbb{C}^k \times \mc{E} (0)$, where the latter has diagonal $K$-action.  
\end{lemma}

\begin{proof} In particular, $\mc{E}$ is a $K$-equivariant holomorphic principal $G$ bundle.
Therefore, by Theorem \ref{hk} $\mc{E}$ admits a $K$-equivariant holomorphic trivialization,
$\mathbb{C}^k \times \mc{E}(0)$. Let $s$ be a holomorphic section of $\mc{E}$, which corresponds to a constant section of the 
trivialization $\mathbb{C}^k \times \mc{E}(0)$. Let $\rho_s: T \to G$ be the local action 
homomorphism corresponding to $s$. Let $x_0$ be any point in $(\mathbb{C}^{*})^k$.  Then 
$$ s(tx_0) = t s(x_0) \cdot \rho_s(t)^{-1}\,. $$ 
Since the homomorphism $\rho_s$ and the $T$-action are algebraic, the holomorphic section $s$ is algebraic over $(\mathbb{C}^{*})^k$.
This implies that $s$ is in fact algebraic over $\CC^k$. Thus $\mc{E}$ admits an algebraic  $K$-equivariant  trivialization. 
\end{proof}

\begin{lemma}\label{lem:cone1}
Let $\mathcal{E}$ be a $T$-equivariant
holomorphic (resp. algebraic) principal $G$-bundle on a nonsingular affine toric variety
$X_{\sigma}$. Then $\mathcal{E} $ is trivial and it admits a
distinguished section.
\end{lemma}

\begin{proof} Recall that $X_{\sigma}\,=\, A_{\sigma}\times O_{\sigma}$, where $A_{\sigma}$ is isomorphic
to $\CC^{\dim{\sigma}}$ and $T_{\sigma}$ acts linearly on $A_{\sigma}$. Then by Theorem \ref{hk}
(resp. Lemma \ref{hka}), the restriction $\mathcal{E}|_{A_{\sigma}}$ is
$T_{\sigma}$-equivariantly isomorphic to $A_{\sigma} \times \mc{E}(p)$, where $p$ is
the $T_{\sigma}$-fixed point of $A_{\sigma}$. Here the $T_{\sigma}$-action on the product $A_{\sigma} \times \mc{E}(p)$ is the natural diagonal
action. Take any $e \in \mc{E}(p)$. Let $\rho: T_{\sigma} \to G$ be the homomorphism induced by the action of $T_{\sigma}$ on $\mc{E}(p)$ defined by
$tz\,=\, z\cdot \rho(t)$.

Let $s$ be the section of $\mathcal{E}\vert_{A_{\sigma}}$ corresponding to the constant section $e$ of
$A_{\sigma} \times \mc{E}(p)$. Then $\rho_s = \rho$ by the definition of $s$.

Let $y_0$ be the identity element in $O_{\sigma}$. We may identify $A_{\sigma}$ with $A_{\sigma} \times\{y_0\}$.
Now extend $s$ to a section of $\mathcal{E}$ over $ X_{\sigma}$ by setting
\begin{equation}\label{eq:extn}
 s(x,k y_0)\,=\, \iota_{\sigma}(k) s(x,y_0)
\end{equation}
for every $x\,\in\, A_{\sigma}$ and $k \,\in\, O_{\sigma}$. This shows that
$\mathcal{E}$ is trivial over $X_{\sigma}$.

The local action function  of the section $s$ over $A_{\sigma}
\times \{y_0\}$ satisfies $$ \rho_{s}((x,y_0),h) =
\rho_{s}((p,y_0),h) =\rho(h) $$
 for all $h \in T_{\sigma}$. Since
$$ \begin{array}{l}
h s(x,k y_0) = h \iota_{\sigma}(k) s(x,y_0) = \iota_{\sigma}(k) h s(x, y_0) \\
 = \iota_{\sigma}(k) s(hx,y_0) \cdot
\rho_{s}((x,y_0),h) = s(hx,ky_0) \cdot \rho_{s}((x,y_0),h),
\end{array} $$
we deduce that
\begin{equation}\label{con}
\rho_{s}((x,ky_0),h) = \rho_{s}((x,y_0),h)
\end{equation}
for every $k\in O_{\sigma}$ and every $h \in T_{\sigma}$.

It follows easily from \eqref{eq:extn} that $\iota_{\sigma}(k) s(x,y) = s(x,ky)$
for any $y \in O_{\sigma}$. Then, using \eqref{con}, we have
$$ \begin{array}{l}
h \iota_{\sigma}(k) s(x,y) = h s(x,ky) = s(hx, ky) \cdot \rho_s((x,ky),h) \\
= s(hx, ky) \cdot \rho_s((x,y_0),h) = s(hx, ky) \cdot
\rho_s((p,y_0),h) \end{array} $$
for every $(h, \iota_{\sigma}(k)) \in T$ and $(x,y) \in A_{\sigma} \times
O_{\sigma} =X_{\sigma} $. Therefore,
$$
\rho_s((x,y), h \iota_{\sigma}(k) ) \,=\, \rho_s((p,y_0), h) \,=\,\rho(h)\, ,
$$
and the lemma follows.
\end{proof}

\begin{lemma}\label{welld} Let $\sigma$ be a cone. Then the homomorphisms $\rho_{s}: T\to G$, induced by different
 distinguished sections $s$ of $\mc{E}\vert_{X_{\sigma}}$, are equal up to conjugation by elements of $G$.
\end{lemma}

\begin{proof} Take a distinguished section $s$. Let $x_{\sigma}=(p,y_0)$ be the fixed point of $T_{\sigma}$
 in $X_{\sigma}$ corresponding to the identity element of $O_{\sigma}$.
 Then for any $x \in X_{\sigma}$,  $\rho_{s}(x,t)\,=\,
\rho_{s}(x_{\sigma},t)$.
  Note that $$\rho_{s}(x,t)\,=\,
\rho_{s}(x_{\sigma},t) =\, \rho_s(x_{\sigma}, \pi_{\sigma}(t)  ) \, , $$ since $s$ is a distinguished section. 
Moreover, 
$$\pi_{\sigma}(t) s( x_{\sigma})   = s ( x_{\sigma}) \cdot \rho_s(x_{\sigma}, \pi_{\sigma}(t)) \,.$$
Therefore, $\rho_s$
is completely determined by $s(x_{\sigma})$ and the canonical
$T_{\sigma}$ action on $\mc{E}(x_{\sigma})$. If $s'$ is another distinguished section, then
there exists $g\,\in\, G$ such that $s'(x_{\sigma}) \,=\, s(x_{\sigma}) \cdot g$. 
 Then $\rho_{s'} \,=\, g^{-1} \rho_{s} g$ follows by applying equation \eqref{eq:rho2} with $x=x_{\sigma}$ , $t \in T_{\sigma}$, and $g(x_{\sigma})= g$.    Note that
by Lemma \ref{ds2} all elements of $G$ appear this way for distinguished sections.
\end{proof}

\begin{defn} For notational convenience, denote by $\rho_{\sigma}$
a homomorphism from $T$ to $G$ induced by a
distinguished section of $\mc{E}\vert_{X_{\sigma}}$. By Lemma \ref{welld}, the
homomorphism $\rho_{\sigma}$ is unique up to conjugation by elements of $G$.
\end{defn}

\begin{lemma}\label{aut} Let $\mc{E}$ be a $T$-equivariant holomorphic (resp. algebraic)  principal $G$-bundle over a nonsingular toric variety $X$ (not necessarily affine).
Let $\Aut_{_{T}}(\mc{E})$ be the group of $T$-equivariant automorphisms
of $\mc{E}$ that project to the identity map on $X$. Then $\Aut_{_{T}}(\mc{E})$  is a subgroup of $G$ that contains the center $Z(G)$ of $G$.
\end{lemma}

\begin{proof}
Fix a point $x_0 \,\in\, O \,\subset\, X$ and also an element $e \,\in\, \mc{E}(x_0)$.
Consider the map
$$\alpha\,:\, \Aut_{_{T}}(\mc{E}) \,\to\, G$$
uniquely determined by the equation $\Phi(e) \,=\, e \cdot \alpha(\Phi)$.
Note that given any $\Phi\,\in\, \Aut_{_{T}}(\mc{E})$, by its
$T$-equivariance and continuity, $\Phi$ is determined by the restriction
$\Phi\vert_{\mc{E}(x_0)}$, while $\Phi\vert_{\mc{E}(x_0)}$ in turn is determined
by $\Phi(e)$ using $G$-equivariance.  Therefore the above map $\alpha$ is injective.
For $\Phi_1\,\in\, \Aut_{_{T}}(\mc{E})$, we have
$$
\Phi_2 \circ \Phi_1 (e) \,=\, \Phi_2 ( e \cdot \alpha(\Phi_1))
\,=\, \Phi_2 (e) \cdot \alpha(\Phi_1)\,=\, e \cdot \alpha(\Phi_2) \alpha(\Phi_1)\, .
$$
This proves that $\alpha$ is a group homomorphism.

Now take any $g \,\in\, Z(G)$. Choose a distinguished section $s$ of
$\mc{E}_{\sigma}$ such that $s(x_0)\,=\,e$. Define $\Phi(e) \,=\, e\cdot
g$. Then $\Phi$ determines  an automorphism of $\mc{E}\vert_O$ using $T$
and $G$ equivariance as follows: $ \Phi( te\cdot h) \,=\, t e \cdot gh$. We need to
check that $\Phi$ extends to the boundary strata of $X$, which consist of
lower dimensional $T$-orbits.
Note that
$$ \Phi(te) \,=\, \Phi(ts(x_0))\,=\, t s(x_0) \cdot g\,=\, s(tx_0) \cdot \rho_s(t) g\, . $$
On the other hand,
$$ \Phi(te)\,= \,\Phi(ts(x_0)) \,=\, \Phi (s(tx_0) \cdot \rho_s(t) )\,  = \, \Phi (s(tx_0)) \cdot \rho_s(t) \, .$$
Therefore,
$$ \Phi(s(tx_0))\,=\, s(tx_0) \cdot \rho_s(t)g \rho_s(t)^{-1} \,=\, s(tx_0) \cdot g$$
as $g \in Z(G)$.
 Any  point $x$ in a boundary stratum can be considered as a limit
$$ \lim_{z\to 0} \lambda(z) x_0\, ,$$ where
$\lambda(z)$ is an algebraic curve in $T$. Now since $s$ and the $G$-action are holomorphic (resp. algebraic),
$\Phi(s(x)) = s(x) \cdot g $ defines an analytic (resp. algebraic) extension of $\Phi$ over $X$.
\end{proof}

\section{Isomorphism classes and admissible collections}\label{sec3}

Let $X$ be a nonsingular toric variety of dimension $n$
corresponding to a fan $\Xi$. Let $\mathcal{E}$ be a holomorphic  (resp. algebraic)
$T$-equivariant principal $G$-bundle over $X$.
Define $\mathcal{E}_{\sigma} \,:=\,
\mathcal{E}\vert_{X_{\sigma}}$ for any cone $\sigma \in \Xi$.

Let $\sigma$ be a maximal
cone in $\Xi$, and let $s_{\sigma}$ be a distinguished section of
$\mc{E}_{\sigma}$.  Let
$\rho_{\sigma}\,:\,T \,\to \,G$ be the corresponding homomorphism.
The group $T$ acts on $X_{\sigma} \times G$ as follows:
\begin{equation}\label{taction}
t(x,h) \,=\, (tx, {\rho}_{\sigma}(t)h )\, .
\end{equation}

Let $\displaystyle{\psi_{\sigma}: \mathcal{E}_{\sigma} \to
X_{\sigma} \times G}$ be the trivialization of $\mc{E}_{\sigma}$ defined by
\begin{equation}\label{tr1}
\psi_{\sigma} (s_{\sigma}(x) \cdot h ) \,=\, (x, h)\, .
\end{equation}
Note that
\begin{equation}\label{tr2}
\psi_{\sigma} (t s_{\sigma}(x) \cdot h ) \,=\,
\psi_{\sigma}(s_{\sigma}(tx) \cdot \rho_{\sigma}(t)h )\,=\, (tx, {\rho}_{\sigma}(t) h\, .
\end{equation}
Hence by \eqref{taction} and \eqref{tr2} we get that $\psi_{\sigma}$ in
\eqref{tr1} is a $T$-equivariant trivialization.

Let $\sigma$, $\tau$ be any two maximal cones. Define
$\phi_{\tau \sigma}\,: X_{\sigma} \cap X_{\tau} \,\to\, G$ as follows:
\begin{equation}\label{phits2}
s_{\sigma}(x) \,=\, s_{\tau}(x) \cdot \phi_{\tau\sigma}(x) \,.
\end{equation}
The maps $\phi_{\tau \sigma}$'s are transition functions for $\mathcal {E}$ as they satisfy
\begin{equation}\label{tran}
\psi_{\tau}\psi_{\sigma}^{-1} (x,h)\,=\, (x, \phi_{\tau \sigma}(x) h)\, .
\end{equation}
By the $T$-equivariance property of $\mathcal E$ we have
\begin{equation}\label{tt3}
t (\psi_{\tau} \psi_{\sigma}^{-1}(x, h)) \,=\,
\psi_{\tau} \psi_{\sigma}^{-1} (t(x,h))\, .
\end{equation}
Using \eqref{taction} and \eqref{tran}, from \eqref{tt3} it follows that
$$ (tx, {\rho}_{\tau}(t) \phi_{\tau \sigma}(x) h ) = (tx, \phi_{\tau \sigma}(tx){\rho}_{\sigma}(t)h ). $$
Therefore, we have $\phi_{\tau \sigma}(tx)\,=\, {\rho}_{\tau}(t) \phi_{\tau \sigma}(x)
{\rho}_{\sigma}(t)^{-1}$.

Fix a point $x_0$ in the principal orbit
$O$. Denote $\phi_{\tau\sigma}(x_0)$ by $P(\tau,\sigma)$. Note that ${\rho}_{\tau}(t) P(\tau, \sigma)
{\rho}_{\sigma}(t)^{-1}$ extends to a holomorphic (resp. algebraic) function, namely $\phi_{\tau\sigma}$, on $X_{\tau \cap \sigma}$.
Moreover, since $\phi_{\tau\sigma}$ are transition functions, $\{P(\tau,\sigma)\}$ satisfy
the cocycle condition
$$ P(\tau,\sigma) P(\sigma, \delta) P(\delta, \tau) \,=\, 1_G\, .$$
Finally, if we change $\{s_{\sigma}\}$ to $\{s_{\sigma} \cdot g_{\sigma}\}$, then
$\{\rho_{\sigma}\}$ transforms to $\{ g_{\sigma}^{-1} \rho_{\sigma} g_{\sigma} \}$, while
$\{P(\tau,\sigma)\}$ transforms to $\{g_{\tau}^{-1} P(\tau,\sigma) g_{\sigma}\}$.

\begin{defn} Let $\Xi^*$ denote the set of maximal cones in $\Xi$. An admissible collection $\{\rho_{\sigma}, P(\tau,\sigma) \}$
consists of a collection of homomorphisms $\{\rho_{\sigma}:T \to G \mid \sigma \in \Xi^* \}$ and a collection of elements
$\{ P(\tau,\sigma) \in G \mid \tau, \sigma \in \Xi^*  \} $ satisfying the
following conditions:
\begin{enumerate} \item $\rho_{\sigma}$ factors through $\pi_{\sigma}: T \to
T_{\sigma}$.

 \item For every pair $(\tau,\sigma)$ of maximal
cones, $ \rho_{\tau} P(\tau, \sigma) \rho_{\sigma}^{-1}$ extends to a $G$-valued
holomorphic (equivalently regular algebraic) function over $X_{\sigma} \cap X_{\tau}$.

\item $P(\sigma, \sigma) = 1_G$ for all $\sigma$.

\item  For every triple $(\tau, \sigma, \delta)$ of
 maximal cones having nonempty intersection, the cocycle condition
$ P(\tau,\sigma) P(\sigma, \delta)$  $P(\delta, \tau) \,=\, 1_G$ holds.
\end{enumerate}

Two such admissible collections $\{\rho_{\sigma}, P(\tau,\sigma) \}$ and
$\{\rho'_{\sigma}, P'(\tau,\sigma) \}$  are equivalent if the following hold:
\begin{enumerate}
\item[(i)] For every $\sigma$ there exists $g_{\sigma}\in G$ such
that $\rho'_{\sigma} \,=\, g_{\sigma}^{-1} \rho_{\sigma} g_{\sigma}$.

\item[(ii)] For every pair $(\tau,\sigma)$, $P'(\tau,\sigma) \,= \,
g_{\tau}^{-1} P(\tau,\sigma) g_{\sigma}$, where $g_{\sigma}$ and $g_{\tau}$ are as in
(i) above.
\end{enumerate}
\end{defn}

\begin{theorem}\label{class} Let $X$ be a nonsingular toric variety and $G$ a complex
linear algebraic group. Then the isomorphism classes of holomorphic (or algebraic) $T$-equivariant
principal $G$-bundles on $X$ are in one-to-one correspondence with equivalence classes of
admissible collections $\{\rho_{\sigma}, P(\tau,\sigma) \}$.
\end{theorem}

\begin{proof}
Let $\mathcal {E}_1$ and $\mathcal E_2$ be two $T$-equivariant principal $G$-bundles
on $X$ which are $T$-equivariantly isomorphic. Fix such an isomorphism $\Phi\,:\,
\mathcal {E}_1\,\to\, \mathcal E_2$. Let $\{s^1_{\sigma}\}$ be a collection of
distinguished sections which associates a
class of admissible collections $\{\rho_{s^{1}_\sigma},  P^{1}(\tau,\sigma) \}$ to 
$\mc{E}_1$. Let
$\displaystyle{s^{2}_{\sigma}\,=\,\Phi \circ s^{1}_{\sigma}}$.  Note that
$\displaystyle{\rho_{s^{1}_\sigma}(x,t)\,=\,\rho_{s^{2}_\sigma}(x,t)}$ by Lemma
\ref{lem:equiv}. Hence the collection $\{s^{2}_{\sigma}\}$ is also distinguished.
Then by $\eqref{phits2}$ and the $G$-equivariance property of $\Phi$ we get that
$$ s^{2}_{\sigma}(x_0)\,=\,\Phi (s^{1}_{\sigma}(x_0) ) \,=\,
\Phi(s^{1}_{\tau}(x_0)) \cdot \phi_{\tau\sigma}(x_0) \,=\,s^{2}_{\tau}(x_0)
\cdot \phi_{\tau\sigma}(x_0)\, .$$
Consequently, $\displaystyle{P^2(\tau,\sigma) \,=\,P^1(\tau,\sigma)}$.
Therefore the admissible collection for
$\mathcal {E}_1$ (with respect to $\{s^1_{\sigma}\}$) coincides with admissible
collection for $\mathcal{E}_2$ (with respect to $\{s^2_{\sigma}\}$).
However, for a different choice of collection of
distinguished sections we get an equivalent admissible collection.

Thus we may associate an equivalence class of
admissible collections to an isomorphism class of equivariant
principal bundles. We will denote this assignment by $\mb{A}$.

Conversely given an admissible collection $\{\rho, P\} = \{\rho_{\sigma},
P(\tau, \sigma) \}$, we construct a principal $G$-bundle $E(\{\rho, P\})$ as follows.
Let $\phi_{\tau \sigma}\,:\, X_{\sigma}\cap X_{\tau} \,\to\, G$ denote
the holomorphic (equivalently regular algebraic) extension of ${\rho}_{\tau} P(\tau,\sigma) {\rho}_{\sigma}^{-1}$.
Note that $\{\phi_{\tau \sigma} \}$
 satisfies the cocycle condition. Therefore we may  construct
 an algebraic (hence, holomorphic) principal $G$-bundle $E(\{\rho, P\}) $ over $X$ with  $\{\phi_{\tau \sigma} \}$
 as transition functions\,:
 $$ E(\{\rho, P\}) = (\bigsqcup_{\sigma} X_{\sigma} \times G )/ \sim$$
where $(x,g) \sim (y, h)$ for $(x,g) \in X_{\sigma} \times G $ and
$(y,h) \in X_{\tau} \times G $ if and only if
\begin{equation}\label{eq:equiv}
x\,=\, y, \ \ x \,\in\, X_{\sigma}\cap X_{\tau} \ \ {\rm and}\ \  h \,=\, \phi_{\tau
\sigma}(x) g\, .
\end{equation}

 The $G$-action on each $X_{\sigma} \times G$ is defined by right multiplication.
It is easy to check that these actions on different charts are compatible.

 Define $T$ action on each $X_{\sigma} \times G$ by
 $t(x,g)= (tx, {\rho}_{\sigma}(t) g)$. Note that if $(y,h) \in X_{\tau} \times
 G$ is equivalent to $(x,g) \in X_{\sigma} \times G $, then
 \begin{equation}\label{eq:equiv2}
t(y,h) = t(x, \phi_{\tau \sigma }(x) g) =
 (tx, {\rho}_{\tau}(t) \phi_{\tau \sigma }(x) g ).\end{equation}
 Now if $x $ belongs to the open orbit $O =T \subset X_{\sigma} \cap
 X_{\tau}$, then
$$
 \phi_{\tau \sigma}(tx) {\rho}_{\sigma}(t)  =
{\rho}_{\tau}(tx) P(\tau, \sigma) {\rho}_{\sigma}(tx)^{-1}
 {\rho}_{\sigma}(t)
 = {\rho}_{\tau}(t) {\rho}_{\tau} (x) P(\tau,\sigma) {\rho}_{\sigma} (x)^{-1}  =
 {\rho}_{\tau}(t) \phi_{\tau \sigma}(x).
$$

 Since both $\phi_{\tau \sigma}(tx){\rho}_{\sigma}(t) $ and $ {\rho}_{\tau}(t) \phi_{\tau
 \sigma}(x)$ are continuous in $x$ on $X_{\sigma} \cap X_{\tau}$
 and $O$ is dense in $X_{\sigma} \cap X_{\tau} $,
\begin{equation}\label{eq:equiv4}
\phi_{\tau \sigma}(tx) {\rho}_{\sigma}(t) \,=\, {\rho}_{\tau}(t)
\phi_{\tau \sigma}(x)\ \ \forall \  x \,\in\, X_{\sigma} \cap
X_{\tau}\, .
\end{equation}
{}From \eqref{eq:equiv2} and \eqref{eq:equiv4}, we have
$$t(y,h) = (tx, \phi_{\tau \sigma}(tx) {\rho}_{\sigma}(t) g   )$$
 whenever $(x,g)\,\sim\, (y,h)$. Since $t(x,g)\,=\, (tx, {\rho}_{\sigma}(t) g)$, using
\eqref{eq:equiv} it follows that $t(y,h)\,\sim\, t(x,g)$ whenever $(x,g) \,\sim\, (y,h)$.
In other words, the $T$-actions on the $X_{\sigma} \times G$
are compatible and define an (algebraic) action of $T$ on $E(\{\rho, P\})$.

It is easy to check that $\mb{A} ([E(\{\rho, P\}) ])\, = \,[\{\rho, P\}] $. This shows
that $\mb{A}$ is surjective.

Now we will show that $\mb{A}$ is injective.
Given a principal $G$-bundle $\mc{E}$ on $X$, consider any representative $\{\rho, P\}$
of $\mb{A}([\mc{E}])$. We will construct an
isomorphism $\Phi: \mc{E} \to E(\{\rho, P\})$. For this first
note that $\{\rho, P\}\, = \,\{ \rho_{\sigma}, P(\tau, \sigma)\}$ is associated to
a collection of distinguished sections $\{s_{\sigma} \}$. Given $e \in \mc{E}$, suppose $e \in \mc{E}_{\sigma}$. Then there
exists a unique $g^{\sigma}_e \in G$ such that
$e\,\,= \,\,s_{\sigma}(\pi(e)) \cdot g^{\sigma}_e$. Define
 $$\Phi(e)\, =\, [(\pi(e), g^{\sigma}_e)],$$ where the right hand side is the equivalence class of
 $ (\pi(e), g^{\sigma}_e)$ under the relation $\sim $ defined in \eqref{eq:equiv}.

Using \eqref{phits2} it is easy to check that $\Phi$ is well-defined. Since
$e\cdot h \,= \,s_{\sigma}(\pi(e)) \cdot g^{\sigma}_e h$, we have
$$\Phi(e\cdot h )\,=\, [\pi(e), g^{\sigma}_e h]\, = \,\Phi(e) \cdot h.$$
Therefore, $\Phi$ is a morphism of principal $G$-bundles. Since every morphism of
principal $G$-bundles over a fixed base is an isomorphism, it follows that $\Phi$ is an isomorphism.
Note that $$te \,=\,t s_{\sigma}(\pi(e)) \cdot g^{\sigma}_e\,=\,
s_{\sigma}(t\pi(e)) \cdot\rho_{\sigma}(t) g^{\sigma}_e\, .$$
Therefore, $$ \Phi(te)\,=\, [(t\pi(e), \rho_{\sigma}(t) g^{\sigma}_e )]
\,=\, t [\phi(e), g^{\sigma}_e] \,=\, t \Phi(e)\ .$$
Hence $\Phi$ is $T$-equivariant. The inverse of $\Phi$ is then automatically $T$-equivariant.

Now suppose $\mathcal E_1$, $\mathcal E_2$ are two equivariant principal $G$-bundle
such that $\mb{A}[\mathcal E_1] \,=\, \mb{A}[\mathcal E_2]\,=\,[\{\rho,P\}]$. Then it
follows from the above that $$\mathcal E_1 \,\cong \, E(\{\rho,P\}) \,\cong\,
\mathcal E_2\, .$$
Therefore $\mb{A}$ is injective. This concludes the proof.
\end{proof}

\begin{remark}\label{algrep} It follows from the proof of Theorem \ref{class}
 that every isomorphism class of 
$T$-equivariant holomorphic principal $G$-bundles over $X$ has an algebraic representative.
\end{remark}

The following generalizes a similar result for abelian structure
groups in \cite{DP}.

\begin{corollary}\label{nil} If $G$ is a nilpotent group, then a
$T$-equivariant holomorphic (or algebraic) principal $G$-bundle over a nonsingular toric
variety admits an equivariant reduction of structure group to a
maximal torus of $G$. In particular, if $G$ is unipotent then the
bundle is trivial with trivial $T$-action.
\end{corollary}

\begin{proof} If $G$ is nilpotent, then it is a direct product $K \times U$, where $K$ is a maximal
torus in $G$ and $U$ is unipotent. Homomorphisms from $T$ to $U$ are trivial. So the image of any
homomorphism $\rho:T \to G$ lies in K.

Recall that if $s_{\sigma}$ is a distinguished section, so is $s_{\sigma} \cdot g$
for any $g \in G$. Using this we can choose distinguished sections so that $P(\tau, \sigma) = 1_G$ for
all pairs $(\tau,\sigma)$ of maximal cones. Then the transitions maps $\phi_{\tau\sigma} = \rho_{\tau} \rho_{\sigma}^{-1}$
take values in $K$. This implies the principal bundle has a reduction to  $K$.

The claim about unipotent $G$ follows easily from the fact that
homomorphisms from $T$ to $G$ are trivial.
\end{proof}

\begin{example} Let $(n_1, \ldots, n_k)$ be a partition of $r$. Let $G$ be the nilpotent subgroup of ${\rm GL}(r, \CC)$
consisting of all block diagonal matrices with $k$ blocks of sizes $n_1, \ldots, n_k$ respectively, where each block is upper
triangular and the diagonal elements within a block are the same. Assume $X$ is a complete, nonsingular toric variety with fan $\Xi$.
 Then the isomorphism classes of $T$-equivariant holomorphic (equivalently algebraic)
 principal $G$-bundles on $X$ are parametrized by $\mathbb{Z}^{dk}$, where
$d$ is the cardinality of $\Xi(1)$. This follows form the fact that equivariant line bundles over $X$ are parametrized by $\mathbb{Z}^d$ (see
 \cite{Oda}).
\end{example}

\section{Kaneyama type description}\label{ktd}

In this section our main goal is to give a more combinatorial flavor to our
classification. We fix an embedding of $G$ in ${\rm GL}(r, \CC)$ such that $K_0 \,
:=\, (\CC^*)^r \bigcap G$ is a maximal torus of $G$. Inspired by the work of Kaneyama,
we restrict to a smaller class of distinguished sections.

\begin{defn}
A distinguished section $s$ of $\mc{E}_{\sigma}$ is called a Kaneyama section if the
 image of the corresponding homomorphism $\rho_{s}$ lies in $K_0$. \end{defn}

\begin{lemma}\label{ks} Every  $\mc{E}_{\sigma}$ has a Kaneyama section.
\end{lemma}

\begin{proof}
Let $s$ be a distinguished section of $\mc{E}_{\sigma}$.
  The image of the corresponding homomorphism $\rho_{s} : T \to G$ is contained in some maximal torus $K_{\sigma}$ of $G$.
    Since all maximal tori in $G$ are
    conjugate, there exists  $h_{\sigma} \in G$ be such that $h_{\sigma} K_{\sigma} h_{\sigma}^{-1} \, =\, K_0$.
    Then $h_{\sigma} \rho_{s}(t) h_{\sigma}^{-1} \in K_0$. Hence $s\cdot h_{\sigma}$ is a Kaneyama section.
    \end{proof}

For a Kaneyama section $s_{\sigma}$, the image $\rho_{\sigma}(t) $ is a diagonal matrix
which we denote by
$$
\xi^{\sigma}(t) \, = \, \text{diag} (\xi^{ \sigma}_1(t)\,,\, \ldots\, ,
\,\xi^{\sigma}_r(t))\, .
$$
Note that for different choices of Kaneyama sections $s_{\sigma}$, the images
$\rho_{\sigma}(t)$ can vary up to conjugation (see Lemma \ref{welld}),
but they are all diagonal matrices.

\begin{lemma}\label{normal} If the maximal torus $K_0$ is normal in $G$, then
any  $T$-equivariant holomorphic (or algebraic) principal $G$-bundle over a nonsingular toric
variety admits an equivariant reduction of structure group to
$K_0$.
\end{lemma}

\begin{proof} Let $\{ s_{\sigma} \}$ be a collection of
Kaneyama sections as $\sigma$ varies over all maximal cones. Let
$\xi^{\tau}(t) P(\tau, \sigma)\xi^{\sigma}(t)^{-1} $ be the
corresponding transition maps. For each $\sigma$, we choose $g_{\sigma} \in G$ such that all $ s_{\sigma}\cdot g_{\sigma}$'s pass through a common point in $\mc{E}(x_0)$. In other words, for these
new sections, all $P(\tau, \sigma) \,= \, 1_G$. The new transition maps
are $ g_{\tau}^{-1} \xi^{\tau}(t) g_{\tau}  g_{\sigma}^{-1}
\xi^{\sigma}(t)^{-1} g_{\sigma}$. These are in $K_0$ as it is normal
in $G$.
\end{proof}




An analogue of the following lemma was proved in \cite{Kan1} in the context of vector bundles.
The idea of our proof is the same as in \cite{Kan1}, but we include it for clarity.
 The result is expected as the $T_{\sigma \bigcap \tau}$-actions on the fibers of $\mc{E}_{\sigma}$ and
  $\mc{E}_{\tau}$ at a fixed point $x_{\sigma \bigcap \tau} \in X_{\sigma } \bigcap X_{\tau}$ of $T_{\sigma \bigcap \tau}$ agree.

\begin{lemma}\label{kl} Suppose $\sigma$ and $\tau$ are  cones of $\; \Xi$. Then there is a permutation
$\gamma$ of $\{ 1, \ldots,r\}$ such that $\eta(\xi^{\sigma}_{i}) = \eta (\xi^{\tau}_{\gamma(i)})$ for every $\eta \in \sigma \bigcap \tau \bigcap \Xi(1)$.
\end{lemma}

 \begin{proof} Recall that the transition function $\xi^{\tau}(t)  P(\tau, \sigma)   \xi^{\sigma}(t)^{-1}$ is regular on
  $ X_{\sigma \bigcap \tau}$. Consider $P(\tau,\sigma)$ as a matrix and
  let $ p_{ij} = P(\tau, \sigma)_{ij}$. Then
  $ p_{ij} \xi^{\tau}_i(t) \xi^{\sigma}_j(t)^{-1} $ is regular on $X_{\sigma \bigcap \tau}$ for each $(i,j)$.

Each $\xi^{\sigma}_i$ may be considered as
  an element of the dual lattice $M$ of $N$.
   Then regularity implies that
  \begin{equation}\label{first} \eta(\xi^{\tau}_i) \ge \eta( \xi^{\sigma}_j ) \end{equation} for every $\eta \in \sigma \bigcap \tau \bigcap \Xi(1)$ whenever $p_{ij} \neq 0$.

    Since
  $$ \det P(\tau,\sigma) = \sum_{\gamma \in S_r} \text{sign}(\gamma) p_{1 \gamma(1)} \ldots p_{r \gamma(r)} \neq 0 \, ,  $$
there exists a permutation $\gamma $ such that $p_{i \gamma(i)} \neq 0$ for every $i$.  Therefore,
$$ \eta(\xi^{\tau}_i) \ge \eta( \xi^{\sigma}_{\gamma(i)} ) $$ for every $i$ and every $\eta \in \sigma \bigcap \tau \bigcap \Xi(1)$.

Moreover since
$$ \det \left[  \xi^{\tau}(t) P(\tau,\sigma ) \xi^{\sigma}(t)^{-1} \right]= \det P(\tau,\sigma) \prod_i \xi^{\tau}_i(t) \xi^{\sigma}_i(t)^{-1}  $$
 is a unit on
$X_{\sigma \bigcap \tau}$, we have
\begin{equation}\label{sec}
\eta(\xi^{\tau}_1 + \ldots \xi^{\tau}_r - \xi^{\sigma}_1 - \ldots \xi^{\sigma}_r)= 0
\end{equation}
for every $\eta \in \sigma \bigcap \tau \bigcap \Xi(1)$. Comparing \eqref{sec} with \eqref{first}, we have
$$ \eta(\xi^{\tau}_i) = \eta(\xi^{\sigma}_{\gamma(i)}) $$
for every $i$  and every $\eta$ in $\sigma \bigcap \tau \bigcap \Xi(1)$.   \end{proof}

Write
$$  \xi^{\sigma} := (\xi^{\sigma}_1, \ldots, \xi^{\sigma}_r) \in M^{\oplus r}.$$
Recall that $\rho_{\sigma}(t)$, and therefore $\xi^{\sigma}(t)$ must factor through the projection
$\pi_{\sigma}: T \to T_{\sigma}$. There is a map $ \pi_{\sigma}^{*}: M_{\sigma} \to M$, associated to $\pi_{\sigma}$, from the characters  of $T_{\sigma}$ to those of $T$. Then, $\xi^{\sigma}$ lies in the 
image of $M_{\sigma}^{\oplus r} $ in $M^{\oplus r}$ under the map induced by $\pi_{\sigma}^{*}$.
Note that this image is same as $M^{\oplus r}$ if $\sigma$ is of top dimension, as $T_{\sigma}=T$
in this case.

Denote the set of maximal cones of $\Xi$ by $\Xi^*$. Consider the following abstract data.
\begin{enumerate}
\item A map $$ \xi: \Xi^* \to ( \pi_{\sigma}^{*} (M_{\sigma}))^{\oplus r}  $$ sending $\sigma $ to $\xi^{\sigma} = (\xi^{\sigma}_1, \ldots, \xi^{\sigma}_r) $
such that for every pair of cones $\sigma, \tau \in \Xi(n)$, there exists a permutation $\gamma$ such that
$$\eta(\xi^{\sigma}_i) = \eta( \xi^{\tau}_{\gamma(i)})$$
for every $i$ and every  $\eta \in \sigma \bigcap \tau \bigcap \Xi(1)$.

\item A map $$P: \Xi^* \times \Xi^* \to G \le GL(r) $$ sending $(\tau,\sigma)$  to $P(\tau, \sigma)$ such that
$P( \tau, \sigma)_{ij} \neq 0$ only if $\eta(\xi^{\tau}_i) \ge \eta(\xi^{\sigma}_j)$ for every $\eta$ in $(\tau \bigcap \sigma) \bigcap \Xi(1)$,
 and such that $P(\sigma, \sigma) =1_G$ and
$$ P ( \tau, \sigma) P(\sigma, \delta)  P (\tau, \delta) = 1_G $$ for every $\tau, \sigma, \delta $ in $\Xi^{*}$.

\item Write $(\xi,P)$ for a collection  $\{\xi^{\sigma}, P(\tau, \sigma) \}$ where $\sigma, \tau$ vary over $\Xi^*$.
Two pairs $(\xi,P)$ and $(\xi',P')$ are said to be equivalent if there exists a permutation $\gamma \in S_r$, depending on $\sigma$, such that
$$ (\xi^{\sigma}_1, \ldots, \xi^{\sigma}_r ) =  (\xi'^{\sigma}_{\gamma(1)}, \ldots, \xi'^{\sigma}_{\gamma(r)} )$$ for every $\sigma \in \Xi^*$,
 and if there exists $$ g: \Xi \to G $$ such that $$ P'(\tau, \sigma)= g(\tau)^{-1} P (\tau, \sigma) g(\sigma) $$ for every
$\tau, \sigma$ in $\Xi^*$.
\end{enumerate}

\begin{theorem}\label{class2} Let $X$ be a  nonsingular toric variety defined by a fan $\Xi$ of dimension $n$. The set
 of isomorphism classes of $T$-equivariant principal holomorphic (or algebraic)
$G$-bundles on $X$ is in bijective correspondence with the set of data $(1)$ and $(2)$ up to the equivalence $(3)$.
\end{theorem}

\begin{proof} Given a principal bundle, a choice of Kaneyama sections $\{s_{\sigma}\}$  associates a set of data
 $(\xi,P)$ satisfying  (1) and (2) to it. Note that $\{\xi_i^{\sigma}: 1\le i \le r\}$  is the set of characters of the representation
$\rho_{\sigma}$. Hence it is invariant under conjugation and therefore well-defined up to permutation under choice of
different Kaneyama sections $s_{\sigma} \cdot g(\sigma)$. We have seen before that such a change would transform $ \{P(\tau,\sigma)\}$
to $\{g(\tau)^{-1} P (\tau, \sigma) g(\sigma)\}$. Same arguments as in the proof of Theorem \ref{class} imply that this association sends an
 isomorphism class of bundles to an isomorphism class of data.

On the other hand, given a set of data  $(\xi,P)$ satisfying (1) and (2), we regard $\xi^{\sigma}$ as $\rho_{\sigma}$ and $P(\tau,\sigma)$ as transition
map at a point $x_0 \in O$ to construct a bundle $E(\{\rho,P\})$. The rest of the proof proceeds as in Theorem \ref{class}.
\end{proof}

Given $\xi$, for every $\sigma \in \Xi(n)$, we have a map
$$m_{\sigma}: \sigma \bigcap \Xi(1) \to \mathbb{Z}^r$$ defined by
$ m_{\sigma}(\eta)_i = \eta(\xi^{\sigma}_i)$. Note that as
$\sigma$ is smooth one may recover $\xi^{\sigma}$ from
$m_{\sigma}$. A strategy of Kaneyama is to reformulate the
equivalence (3) in terms of $m_{\sigma}$. It turns out as follows.

(3$'$) Two pairs $(\xi,P)$ and $(\xi',P')$ are said to be equivalent if there exists a permutation $\gamma= \gamma(\sigma) \in S(r)$ such that
$$ (m_{\sigma}(\eta)_1, \ldots, m_{\sigma}(\eta)_r ) =  (m_{\sigma}'(\eta)_{\gamma(1)}, \ldots, m_{\sigma}'(\eta)_{\gamma(r)} )$$ for every $\sigma \in \Xi(n)$ and
$\eta \in \sigma \bigcap \Xi(1)$, and if there exists $$ g: \Xi(n) \to G $$ such that $$ P'(\tau, \sigma)= g(\tau)^{-1} P (\tau, \sigma) g(\sigma) $$ for every
$\tau, \sigma$ in $\Xi(n)$.

The proof of the following result is straight-forward.

\begin{theorem}\label{class3} Let $X$ be a  nonsingular toric variety defined by a fan $\Xi$ of dimension $n$, such that every maximal cone of $\Xi$ has dimension $n$. Then,
the set of isomorphism classes of $T$-equivariant holomorphic (or algebraic) principal 
$G$-bundles on $X$ is in bijective correspondence with the set of data $(1)$ and $(2)$ up to the equivalence $(3')$.
\end{theorem}

\begin{remark}\label{srb} A conjecture of Hartshorne \cite{Har} says that all vector bundles
 of rank two over ${\mathbb P}^n$ should split if $n \ge 7$. This is still open.
In \cite{Kan2}, the following equivariant analogue is discussed:
Any $T$-equivariant vector bundle of rank $r$ on   ${\mathbb P}^n$
splits if $r < n$. An attempt to reproduce the arguments given there
runs into the problem that conjugation by an elementary matrix does not preserve
the diagonal matrices.
A stronger result has recently been proved in an
article of Ilten and S$\ddot{\text{u}}$ss \cite{IS} using Klyachko-type filtrations \cite{Kly}.
It would be interesting to investigate if their result can be generalized to equivariant principal bundles.
\end{remark}

{\bf Acknowledgement.} It is a pleasure to thank V. Balaji for telling us about the work of
 Heinzner and Kutzschebauch, and Pralay Chatterjee for many useful discussions.
 We thank an anonymous referee for his/her valuable comments on an earlier draft.
  The first-named author is supported by a J. C. Bose fellowship. The last-named author is supported by a FAPA grant from the
Universidad de los Andes.

\end{document}